\documentclass[a4paper,leqno]{amsart}

\usepackage{amsmath}
\usepackage{amsthm}
\usepackage{amssymb}
\usepackage{amsfonts}
\usepackage[applemac]{inputenc}
\usepackage{mathrsfs}

\def\R{{\mathbb R}}
\def\N{{\mathbb N}}

\def\virgp{\raise 2pt\hbox{,}}

\def\({\left(}
\def\){\right)}
\def\<{\left\langle}
\def\>{\right\rangle}
\def\le{\leqslant}
\def\ge{\geqslant}

\def\d{{\partial}}

\def\eps{\varepsilon}
\def\l{\lambda}
\def\om{\omega}

\newcommand{\de}{\mathrm{d}}
\newcommand{\re}{\mathrm{Re}}
\newcommand{\im}{\mathrm{Im}}

\newcommand{\mez}{\frac{1}{2}}
\newcommand{\1}[1]{\frac{1}{#1}}
\renewcommand{\d}{\partial}

\newcommand{\conj}[1]{\overline{#1}}
\newcommand{\modu}[1]{\left|#1\right|}
\newcommand{\pare}[1]{\left(#1\right)}

\DeclareMathOperator{\diver}{div}

\newcommand{\D}{d}

\theoremstyle{plain}
\newtheorem{theorem}{Theorem}[section]
\newtheorem{lemma}[theorem]{Lemma}
\newtheorem{corollary}[theorem]{Corollary}
\newtheorem{proposition}[theorem]{Proposition}

\theoremstyle{definition}
\newtheorem{definition}[theorem]{Definition}

\newtheorem{remark}[theorem]{Remark}
\newtheorem*{remark*}{Remark}

\numberwithin{equation}{section}

\begin{document}

\title[Cubic NLS with nonlinear dissipation]
{Global well-posedness for cubic NLS with nonlinear damping}
\author[P. Antonelli]{Paolo Antonelli}
\address[P. Antonelli]{Department of Applied Mathematics and
Theoretical Physics\\
CMS, Wilberforce Road\\ Cambridge CB3 0WA\\ England}
\email{p.antonelli@damtp.cam.ac.uk}
\author[C. Sparber]{Christof Sparber}
\address[C. Sparber]{Department of Applied Mathematics and Theoretical
Physics\\
CMS, Wilberforce Road\\ Cambridge CB3 0WA\\ England}
\email{c.sparber@damtp.cam.ac.uk}
\begin{abstract}
We study the Cauchy problem for the cubic nonlinear Schr\"odinger equation, perturbed by (higher order) dissipative nonlinearities.
We prove global in-time existence of solutions for general initial data in the energy space. In particular we treat the 
energy-critical case of a quintic dissipation in three space dimensions.
\end{abstract}

\date{\today}


\thanks{This publication is based on work supported by Award No. KUK-I1-007-43, funded by the King Abdullah University of Science and Technology (KAUST). The authors
thank the Institute for Pure and Applied Mathematics (Los Angeles) for its hospitality and financial support}
\maketitle

\section{Introduction}
\label{sec:intro}

In this paper we study the Cauchy problem of the cubic nonlinear Schr\"odinger equation (NLS) with nonlinear damping, i.e.
\begin{equation}\label{eq:diss_NLS}
\left \{
\begin{split}
i\d_t u  + \mez\Delta u = & \, V(x) u+  \lambda|u |^2 u -i\sigma| u |^{p-1}u , \quad (t,x) \in [0, \infty)\times \R^d,\\
u(0) = & \, u_0(x),
\end{split}
\right.
\end{equation}
with given parameters $\lambda \in \R$ and $\sigma \ge 0$, the latter describing the strength of the dissipation within our model.
We shall consider the physically relevant situation of
$d\le 3$ spatial dimensions and assume that the dissipative nonlinearity is at least of the same order as the cubic one, i.e. $p\ge 3$.
However, in dimension $d= 3$, we shall restrict ourselves to $3 \le p \le 5$.
In other words, we assume that the dissipative effects
in our model can be described by nonlinearities, which are at most energy critical in the sense of \cite{Caz, T}.
In the following, we shall also assume the external potential $V$ to be an anisotropic quadratic confinement, i.e.
\begin{equation}\label{eq:pot}
V(x)= \frac{1}{2}\sum_{j=1}^d \om_j^2x_j^2, \quad \omega_j \in \R.
\end{equation}
In the case without any dissipation, i.e. $\sigma = 0$, equation \eqref{eq:diss_NLS} simplifies to the classical cubic NLS, a canonical description for
weakly nonlinear wave propagation in dispersive media \cite{SuSu}. It arises in various areas of physics, such as nonlinear optics, plasma physics
or ultracold quantum gases (Bose-Einstein condensates), cf. \cite{SuSu} for a broader introduction. 
Due to the inclusion of a quadratic potential, the natural energy space when studying the Cauchy problem for this equation is given by
\begin{equation*}
\Sigma:= \big \{ u \in H^1(\R^d) \ : \ x u \in L^2(\R^d) \big \},
\end{equation*}
and we consequently denote
\begin{equation*}
\| u \|_{\Sigma}:=\| u \|_{L^2}+\|\nabla u \|_{L^2} + \| x u \|_{L^2}.
\end{equation*}
In the case without harmonic confinement, the results on local and global well-posedness properties of (purely dispersive) cubic NLS are by now considered to be classical, see e.g. \cite{Caz, T}
and the references given therein. In particular, it is very well known that in the \emph{focusing case} $\lambda < 0$ 
finite-time blow-up of solutions can occur if $d\ge 2$, i.e.
$$\lim_{t \to T} \| \nabla u (t, \cdot) \|_{L^2(\R^d)} = \infty,$$ where the blow-up time 
$T < \infty$ depends on the initial data $u_0 $. Generalization to the case with harmonic potentials have been studied in \cite{Fu, Oh, Zh} and more recently in 
\cite{Car1, Car3, Car, Carles}. For confining potentials, the well-posedness results are found to be very much the same as in the classical situation (see also the discussion in Remark \ref{rem: blow-up} below).

From the point of view of physics, the occurrence of blow-up usually implies that new effects have to be taken into account in order to extend the model beyond 
the appearing singularities. As a specific
example, let us briefly describe such a situation in the context of \emph{collapsing Bose-Einstein condensates}: 
These are ultracold quantum gases, which (in a mean-field approximation) can be described
via the Gross-Pitaevskii equation, a cubic NLS governing the macroscopic wave function of the condensate.
In addition, one usually takes into account a harmonic confinement $V$, modeling
the electromagnetic trap present in actual physical experiments. It is nowadays possible to study 
the collapse of the condensate experimentally by tuning the effective nonlinear interaction from positive to negative. 
What one observes is, that as the particle density
increases around the blow-up point, atoms are suddenly emitted from the condensate in bursts, so-called \emph{jets}. 
These jets are caused by (inelastic) \emph{three-body interaction} or recombination effects taking place only at high densities \cite{A}.
Mathematically, three-body forces can be effectively described by quintic nonlinearities, cf. \cite{CP}. 
In order to describe the emittance of particles from the condensate within the realm of Gross-Pitaevskii theory one usually employs the following dissipative model (see \cite{A, KMS, SU} and the references given therein):
\begin{equation}\label{eq:GP}
i \hbar \d_t u + \frac{\hbar^2}{2m} \Delta u = V(x) u + g  |u |^2 u - i  \sigma_3   | u |^{4}u .
\end{equation}
Here $g = 4\pi N \hbar^2  a /m $, with $m$ denoting the mass of the particles and $a$ their scattering length. The 
parameter $\sigma_3$ denotes the three particle recombination loss-rate. Rescaling \eqref{eq:GP} into dimensionless form 
yields \eqref{eq:diss_NLS} with $p=5$ and $d=3$. Note that in \eqref{eq:GP} the total mass, i.e. $M(t):= \| u(t, \cdot) \|^2_{L^2}$, is no longer conserved, 
as can easily be seen from the dissipation equation for the particle density $\rho= |u|^2$:
\begin{equation}\label{eq:diss_rho}
\d_t\rho+\diver J=-2\sigma\rho^{3},
\end{equation}
where $J:=\im(\conj{u}\nabla u)$ denotes the current density. Numerical simulations show \cite{BJM} (see also \cite{BJ}),
that the solution to \eqref{eq:GP} undergoes a series of events where concentration is
followed by consequent dissipation, lowering the density and hence preventing blow-up. We consequently expect that equation \eqref{eq:GP} and more generally \eqref{eq:diss_NLS},
admits global in-time solutions (even for large initial $u_0 \in \Sigma$) and it is the main purpose of our work to prove
that this is indeed the case. Obviously, we are mainly interested in the focusing case $\lambda < 0$,
but we shall keep $\lambda \in \R$ for the sake of
generality. Let us also remark that, even though equation \eqref{eq:GP} has been the main motivation for the 
present work, similar NLS type models with nonlinear damping terms also appear in other areas of 
physics, see e.g. \cite{Bi, PSS, PeSt, SaMa}. For example, in the context of nonlinear optics, nonlinear damping terms are used to describe 
multi-photon absorption, cf. \cite{Fi}.

From a mathematical point of view, we are facing two basic problems in setting up a global existence theory for \eqref{eq:diss_NLS}: 
First, in the physically most relevant
situation where $p=5$ and $d=3$, the (quintic) dissipative nonlinearity is known to be energy-critical \cite{Caz, T} and thus cannot
be considered as a small perturbation of the Laplacian any more. Indeed, global well-posedness for defocusing energy critical NLS
in $d=3$ dimensions (without external potentials) has only been proved recently in the seminal work of Colliander et al. \cite{CKSTT} (see also \cite{KiVi}). 
Even though our nonlinearity is dissipative, this property is not seen locally in time (e.g. in terms of Strichartz-estimates) and we are therefore 
in a similar situation as for the Hamiltonian case. 
A second obstruction for a global existence-theory, even for energy sub-critical cases, is the lack of conserved quantities, 
in contrast to the usual case of Hamiltonian NLS. This is also the main difference to the work of Tao et al. \cite{TVZ}, which treats purely dispersive 
NLS with combined power type nonlinearities. Also there, blow-up is prevented by sufficiently strong (i.e. higher order) defocusing nonlinearities. Note however, 
that such a (mass conservative) model would not be able to describe the coherent loss of particles found in collapsing Bose-Einstein condensates.

Before going further, let us briefly compare our situation to the mathematically much better studied case of \emph{linearly damped} NLS, i.e. 
\begin{equation}\label{eq:lin_NLS}
i\d_t u + \mez\Delta u  = V(x) u + \lambda|u |^2 u - i\sigma u.
\end{equation}
There exists a considerable amount of results in the physics and
mathematics literature for such (weakly damped) NLS type equations. In particular the Cauchy problem to \eqref{eq:lin_NLS} has been analyzed in \cite{OT, Ts1} and more detailed
properties concerning the long-time behavior of solutions can be found in e.g. \cite{Fi, Go, La}. The main difference between our case and \eqref{eq:lin_NLS} is,
that the latter can be treated by the phase-transformation $u(t,x) \to e^{\sigma t} u(t,x)$, which makes the damping term vanish.

In contrast to the situation with linear damping,
the literature on NLS with nonlinear damping is not so abundant: In \cite{KiShi, Shi} 
the asymptotic behavior of \emph{small} solutions to NLS with dissipative nonlinearities 
of the form $\lambda |u|^{p-1}u$, with $\im \, {\lambda \ge 0}$ and $1<p\le 1+2/d$, is studied.
For higher order dissipative nonlinearities, the only rigorous results, we are aware of, 
are given in \cite{PSS}, where, for $p=2(1+s)$ with $s>0$, the authors prove a non-uniform (in-time) bound on the $H^1$
norm of the solution. Numerical studies of NLS type models with nonlinear damping can be found in \cite{BJ, BJM} and we also mention 
the results of \cite{Fi}, based on modulation theory. 

In the following section we shall present our main theorem and deduce from it several corollaries. 
The corresponding proofs are then given in Section \ref{sec:critical}  (where the energy critical case is treated)
and Section \ref{sec:sub}.

\section{Main results} \label{sec:results}

Our first main result concerns the case of an energy-critical damping term.

\begin{theorem}\label{th:critical}
Let $d=3$ and $p = 5$. Assume $V$ to be a quadratic confinement of the type \eqref{eq:pot} and $u_0 \in \Sigma$. Then, for any $\lambda \in \R$ and for any $\sigma >0$ the equation \eqref{eq:diss_NLS} has a unique
global in-time solution $u \in C([0, \infty), \Sigma)$, such that
$$
\int_0^\infty \int_{\R^3}|u(t,x)|^{10} d x \, d t \le C( \|u_0\|_{\Sigma}).
$$
\end{theorem}
The theorem confirms the numerical results of \cite{BJM} on the time-evolution of attractive Bose-Einstein condensates with three-body recombination. No collapse, i.e. finite-time blow-up, occurs due to the dissipation. 
Note that Theorem \ref{th:critical} holds for any $\sigma>0$, no matter how small, which is consistent with the physics literature,
where, after a dimensionless rescaling, 
one finds $\sigma \ll 1$ for any realistic parameter regime \cite{A, BJM, KMS}.
The additional a-priori estimate on the $L^{10}_{t,x}$ norm of the solution (to be proved in Proposition \ref{prop:a-priori}) is reminiscent of the one used in \cite{CKSTT}.
This space-time bound is also required in our proof, since it is well known (see e.g. \cite{Caz}) that
the usual a-priori bound on the $H^1$ norm is not be sufficient to conclude global existence. The reason being that the local existence-time of solutions does not only depend 
on the $H^1$ norm of $u$, but also on its profile.

As a first consequence of the above given theorem, we obtain the analogous statement for all energy subcritical situations (in $d\le 3$)
where the nonlinear damping term is of higher order than cubic.

\begin{corollary}\label{cor:sub}
Let $V$ be a quadratic confinement and $u_0 \in \Sigma$. Moreover, let $p > 3$, if $d=1,2$ and $3< p < 5$, if $ d=3$. Then, for any $\lambda \in \R$ and any $\sigma >0$, the equation \eqref{eq:diss_NLS} has a unique
global in-time solution $u \in C([0, \infty), \Sigma)$.
\end{corollary}
The result strengthens the one of \cite{PSS}, as we are able to prove uniform (in-time) a-priori bounds
on the energy. As a by-product of our analysis we also obtain a well-posendess result in the case where the damping term 
is not of higher order but given by a cubic nonlinearity, i.e. we consider
\begin{equation}\label{eq:diss_cubic}
i\d_t u + \mez\Delta u = V(x) u +(\lambda -i\sigma)|u |^2 u .
\end{equation}
In the context of Bose-Einstein condensates, equation \eqref{eq:diss_cubic} corresponds to a model 
where one only takes into account two-body losses, see \cite{A, SU}.
Most of the time, however, they are neglected in view of the three-body recombination effects.
\begin{corollary}\label{cor:cubic}
Let $d\le 3$, assume $V$ to be quadratic and $u_0 \in \Sigma$. Then, for any $\sigma \ge \text{\rm max}\, \{0, -\lambda\}$ the equation \eqref{eq:diss_cubic} has a unique
global in-time solution $u \in C([0, \infty), \Sigma)$.
\end{corollary}
Thus, even in the case where the two nonlinearities exactly balance, i.e. $\sigma = |\lambda|$, we have a global in-time solution. 
This indicates that dissipation acts on a faster time-scale than the focusing nonlinearity does.
In situations where $\lambda <0$ (focussing case) and $\sigma < |\lambda|$ we expect that solutions to \eqref{eq:diss_cubic} in general exhibit finite-time blow-up (see also the discussion in 
Remark \ref{rem: blow-up}). 
It remains an interesting open problem to rigorously prove that this is indeed the case. 
\begin{remark} In the case where $V (x) \equiv 0$, equation \eqref{eq:diss_cubic} is similar to the \emph{complex Ginzburg-Landau equation}, which usually reads
\begin{equation}\label{eq:GL}
\d_t u = (1 + i \alpha)\Delta u  - (1 + i\beta)|u |^2 u + \gamma u,
\end{equation}
for some given parameters $\alpha, \beta, \gamma \in \R$, see e.g. \cite{AK} for a broader introduction. In the case of no external driving field $\gamma = 0$, the main difference 
between the nonlinearly damped cubic NLS \eqref{eq:diss_cubic} and the
Ginzburg-Landau equation \eqref{eq:GL} is that the latter invokes an additional linear diffusion $\propto  \Delta u$. Equation \eqref{eq:diss_cubic} can therefore be considered as the  
diffusionless limit of \eqref{eq:GL} and our well-posedness results can consequently be reinterpreted within this limiting regime of the complex Ginzburg-Landau equation. For an application 
of this model in the context of nonlinear optics see \cite{SaMa}.
\end{remark}
We finally state the following result on time-decay of solutions as $t\to + \infty$.
\begin{corollary}\label{cor:decay}
Let $u \in C([0, \infty), \Sigma)$, be a global-in-time solution to \eqref{eq:diss_NLS}.
Then $u(t)$ decays to zero as $t \to + \infty$, in the following sense:
For each sequence of time-steps $(t_n)_{n\in \N}$ tending to infinity, there exists a subsequence $(t_{n_\ell})$, such that
\begin{equation}\label{decay}
\lim_{\ell \to \infty} \|u(t_{n_\ell}) \|_{L^2(\R^d)} = 0.
\end{equation}
Moreover, let $u_n(t,x):=u(t+t_n,x)$, then, up to extraction of subsequences, we also have $u_n \to 0$ in $ L^2(0,T);H^1_{\rm loc}(\R^d))$.
\end{corollary}
Unfortunately, we are not able so far to derive explicit rates for the time-decay of the (total mass of the) solution.
Note that the decay rates obtained in \cite{KiShi, Shi} in general do not hold in our case,
since we do not restrict ourselves to small solutions.

\begin{remark} In the fluid dynamics picture of quantum mechanics the equation for the mass density $\rho=|u|^2$ has to 
be complemented by an equation for $J=\im(\conj{u}\nabla u)$, forming the so-called \emph{quantum hydrodynamic system}:
\begin{equation}\label{eq:QHD}
\left \{
\begin{split}
& \,  \d_t\rho+\diver J=-2\sigma\rho^{(p+1)/2},\\
& \,  \d_t J +\diver \left( \frac{J \otimes J}{\rho} \right) + \rho \nabla V + \frac{\lambda}{2}\, \nabla \rho^2  =   \frac{1}{2}\, \rho \nabla \left( \frac{\Delta \sqrt \rho} {\sqrt \rho} \right) - 2\sigma\rho^{(p-1)/2}J.
\end{split}
\right.
\end{equation}
Observe that the equation for the current density $J$ also picks up a nonlinear damping term.
The connection between \eqref{eq:diss_NLS} and \eqref{eq:QHD} can be established following \cite{AnMa, GaMa}, in order to translate our well-posedness results
for \eqref{eq:diss_NLS} into the analogous statements for \eqref{eq:QHD}.
\end{remark}

\section{The energy critical case}\label{sec:critical}

This Section is devoted to the Proof of Theorem \ref{th:critical}. That is, we want to show that solutions to
\begin{equation}\label{eq:quintic_NLS}
i\d_t u  + \mez\Delta u =  V(x) u+   \lambda|u |^2 u -i\sigma| u |^{4}u , \quad (t,x) \in [0, \infty)\times \R^3,
\end{equation}
with $u(0) = u_0(x) \in \Sigma$, exist globally in-time. To this end, we shall first derive 
several a-priori bounds on the solution.

\subsection{A-priori estimates}\label{sec:a-priori} From the dissipation equation \eqref{eq:diss_rho}, we immediately obtain
the following lemma.
\begin{lemma}  \label{lem:a-priori}
Let $u(t) \in \Sigma$ be a solution of \eqref{eq:quintic_NLS}. Then it holds
\begin{equation}\label{eq:a-priori1}
\| u(t, \cdot) \|_{L^2} \le \| u_0 \|_{L^2},\quad \forall \, t \ge 0,
\end{equation}
and in addition we have
\begin{equation}\label{eq:a-priori2}
\int _0^\infty \| u(t, \cdot) \|_{L^6}^6\, dt \le C (\| u_0 \|_{L^2}) .
\end{equation}
\end{lemma}
\begin{proof} We argue formally and first integrate \eqref{eq:diss_rho} w.r.t. $x$ to obtain
$$
\frac{d}{dt} \, \| u(t, \cdot) \|_{L^2}^2 = - 2 \sigma \int _{\R^3} | u(t,x) |^6dx \le 0,
$$
which consequently yields \eqref{eq:a-priori1}. If we then also integrate w.r.t. $t$, we get
\begin{equation*}
2\sigma\int_0^\infty \int_{\R^3} | u(t,x) |^6 d x\, d t=\int_{\R^3} | u(0,x) |^2 d x-
\int_{\R^3}| u(t,x) |^2 \, d x\le \int_{\R^3} | u(0,x) |^2 d x
\end{equation*}
and thus \eqref{eq:a-priori2}. By a standard density argument this a-priori estimates are easily 
shown to hold for any $u(t, \cdot) \in \Sigma$.
\end{proof}

To proceed further consider the following energy-type functional
\begin{equation}\label{eq:diss_NLS_en}
E_\kappa(t):=\int_{\R^3}\frac{1}{2}\, |\nabla u(t,x)|^2+ V(x) |u(t,x)|^2 + \frac{\lambda}{2}\, |u(t,x)|^4+\kappa|u(t,x)|^6\de x,
\end{equation}
for some parameter $\kappa  >0$ to be chosen later on.
The following a-priori bound on this functional will be a key ingredient in the proof of Theorem \ref{th:critical}.

\begin{proposition}  \label{prop:a-priori}
Let $u(t) \in \Sigma$ be a solution of \eqref{eq:quintic_NLS} and $V(x)$ a quadratic confinement of the form \eqref{eq:pot}. 
Moreover, let $0 < \kappa<\sigma / 6$. Then, it holds
\begin{equation*}
E_\kappa(t)\le E_\kappa(0)+C (\|u_0\|_{L^2}), \quad \forall \, t\ge 0,
\end{equation*}
where $C= C(\kappa, |\lambda|, \sigma)>0$ and we also have the following space-time bounds
\begin{align*}
\int_0^\infty\int_{\R^3}|u(t, x)|^4|\nabla u(t,x)|^2\, d x  \, d t\le & C(E_\kappa(0), \|u_0\|_{L^2}),\\
\int_0^\infty\int_{\R^3}V(x)|u(t, x)|^6\, d x\, d t\le & \, C(E_\kappa(0), \|u_0\|_{L^2}).
\end{align*}
In particular, it also holds
\begin{align*}
\int_0^\infty\int_{\R^3}|u(t,x)|^{10}\, d x\,  d t\le  \, C(E_\kappa(0), \|u_0\|_{L^2}).
\end{align*}
\end{proposition}
\begin{proof} Let us assume $u$ is regular enough to perform the following formal manipulations.
A standard density argument can then be used to justify all of them rigorously.
Computing the time-derivative of $E_\kappa(t)$, we obtain
\begin{align*}
\frac{d}{dt} \, E_\kappa(t)
=&\ \sigma\int_{\R^3}|u|^4\re(u\Delta\conj{u})\,dx-2\sigma\int_ {\R^3}V|u|^6\,dx-2\sigma\lambda\int_{\R^3}|u|^8\,dx\\
&-3\kappa\int_{\R^3}|u|^4\im(\conj{u}\Delta u)\,dx-6\kappa\sigma\int_{\R^3}|u|^{10}\, dx\\
=&-\sigma\int_{\R^3} \rho|\nabla\rho|^2\, d x-\sigma\int_{\R^3}\rho^2|\nabla u|^2\, d x-2\sigma\lambda\int_{\R^3}\rho^4\, d x\\
&-3\kappa \int_{\R^3}\rho^2\im(\conj{u}\Delta u)\, d x-6\sigma \kappa \int_{\R^3}\rho^5\, d x - 2\sigma \int_{\R^3} V(x) \rho^3 \, dx\\
=&-\sigma\int_{\R^3}\rho|\nabla\rho|^2\, d x-\sigma\int_{\R^3}\rho^2|\nabla u|^2\, d x-2\sigma\lambda\int_{\R^3}\rho^4\, d x\\
&+6\kappa \int_{\R^3}\rho\nabla\rho\cdot \im(\conj{u}\nabla u) \, d x-6\sigma \kappa \int_{\R^3}\rho^5\, d x - 2\sigma \int_{\R^3} V(x) \rho^3 \, dx,
\end{align*}
where for the last equality we perfom an integration by parts, using the fact that
$\im(\conj{u}\Delta u)=\diver\pare{\im(\conj{u}\nabla u)}$.
Next, we note that $\nabla\rho=2\re(\conj{u}\nabla u)$ and thus
\begin{equation*}
\rho|\nabla u|^2 =|\conj{u}\nabla u|^2= |\re(\conj{u}\nabla u)|^2+|\im(\conj{u }\nabla u)|^2 = \1{4}|\nabla\rho|^2+|J|^2,
\end{equation*}
where, as before, $J=\im(\conj{u}\nabla u)$ denotes the current density. Using this, we can rewrite
\begin{align*}
\int_{\R^3}\rho\nabla\rho\cdot\im(\conj{u}\nabla u) \, d x
=& -\int_{\R^3}\rho\modu{\mez\nabla\rho-J}^2\, d x
+\int\rho\pare{\1{4}|\nabla\rho|^2+|J|^2}\, d x\\
=& -\int_{\R^3}\rho\modu{\mez\nabla\rho-J}^2\, d x
+\int_{\R^3}\rho^2|\nabla u|^2\, d x.
\end{align*}
This consequently yields
\begin{equation} \label{eq:E}
\begin{split}
\frac{d}{dt} \, E_\kappa(t)
=&-\sigma\int_{\R^3}\rho|\nabla\rho|^2\, d x-(\sigma-6\kappa)\int_{\R^3}\rho^2|\nabla u|^2\, d x-2\sigma\lambda\int_{\R^3}\rho^4\, d x\\
&-6\kappa \int_{\R^3}\rho\modu{\mez\nabla\rho-J}^2\, d x -6\sigma \kappa \int_{\R^3}\rho^5\, d x - \sigma \int_{\R^3} V(x) \rho^3 \, dx.
\end{split}
\end{equation}
Under the assumption $\kappa< \sigma /6$, all the terms on the right hand side are negative
(recall that $V(x)$ is assumed to be a quadratic confinement), except for the one
proportional to $\|\rho \|_{L^4}$, which we shall treat by interpolation. Since
\begin{equation*}
\|\rho\|_{L^4}\le\|\rho\|_{L^3}^{3/8}\, \|\rho\|_{L^5}^{5/8},
\end{equation*}
we have
\begin{equation*}
\int_{\R^3}\rho^4 \, d x\le\1{2\eps}\int_{\R^3}\rho^3\, d x+\frac{\eps}{2}\int_{\R^3}\rho^5\, d x,
\end{equation*}
for some arbitrary constants $\eps >0$. Choosing $\eps < 6 \kappa / |\lambda|$, we consequently obtain
\begin{equation} \label{eq:E1}
\begin{split}
\frac{d}{dt} \, E_\kappa(t)
\le &-\sigma\int_{\R^3}\rho|\nabla\rho|^2\, d x-(\sigma-6\kappa)\int_{\R^3}\rho^2|\nabla u|^2\, d x - \sigma \int_{\R^3} V(x) \rho^3 \, dx\\
&-6\kappa \int_{\R^3}\rho\modu{\mez\nabla\rho-J}^2\, d x -C_1(\eps) \int_{\R^3}\rho^5\, d x + C_2(\eps) \int_{\R^3}\rho^3\, d x.
\end{split}
\end{equation}
There still remains a positive term on the r.h.s., namely the last one. However, we already know from \eqref{eq:a-priori2} that
\begin{equation*}
\int_0^T\int_{\R^3}\rho^3\, d x \, d t\le C (\| u_0 \|_{L^2}).
\end{equation*}
We can therefore integrate \eqref{eq:E1} w.r.t. time and using the fact that $E_\kappa(0)< \infty$ by assumption, we
obtain the assertion of the proposition, provided $\kappa< \sigma /6$.
\end{proof}
\begin{remark} The proof shows, that it would not be sufficient to consider only the energy functional
of the unperturbed equation, i.e.
\begin{equation*}
E_0(t)=\int_{\R^3}\frac{1}{2}\, |\nabla u(t,x)|^2+ V(x) u(t,x) + \frac{\lambda}{2}\, |u(t,x)|^4.
\end{equation*}
In fact, computing the
time-derivative of $E_0(t)$, we find
\begin{equation*}
\frac{1}{\sigma}\, \frac{d}{d t} \, E_0(t)=- \int_{\R^3}\rho|\nabla\rho|^2\, d x- \int_{\R^3}\rho^2|\nabla u |^2\, d x
-2 \int_{\R^3}V \rho^3\, d x
-2\lambda \int_{\R^3}\rho^4\, d x.
\end{equation*}
Thus, if the $\lambda <0$ (focusing cubic nonlinearity) the last term on the r.h.s.
is positive and hence we can not conclude that $E_0(t)$ is non increasing.
\end{remark}

As an immediate corollary of Proposition \ref{prop:a-priori} we obtain the uniform boundedness of the energy norm of $u(t)$.

\begin{corollary} \label{cor:a-priori}
Let $V$ be quadratic and $u(t)$ a solution to \eqref{eq:quintic_NLS}. Then
\begin{equation*}
\| u(t, \cdot) \|_{\Sigma} \le C(\| u_0 \|_{\Sigma}),\quad \forall \, t \ge 0.
\end{equation*}
\end{corollary}
\begin{proof} We already know from Lemma \ref{lem:a-priori} that the $L^2$ norm of $u(t)$ is bounded. If $\lambda >0$ then the uniform bound on $E_\kappa(t)$
immediately yields the assertion. On the other hand, if $0>\lambda = - |\lambda|$, we write
\begin{equation*}
\|\nabla u(t, \cdot)\|_{L^2}^2 + \| x  u(t, \cdot)\|_{L^2}^2  \le 2 E_0(t)+\frac{|\lambda|}{2}\|u (t, \cdot)\|_{L^4}^4,
\end{equation*}
since $V$ is quadratic. Now, by interpolation $\|u \|_{L^4}^4\le \| u \|_{L^2}\| u \|_{L^6}^3$, and thus
\begin{equation*}
\|\nabla u(t, \cdot)\|_{L^2}^2+ \| x  u(t, \cdot) \|_{L^2}^2 \le 2 E_0 (t)+\frac{|\lambda|}{4\eps}\|u (t, \cdot)\|_{L^2}^2
+\frac{ |\lambda| \eps}{4}\|u(t, \cdot)\|_{L^6}^6
\end{equation*}
for some $\eps > 0$. If we then choose $\eps=\frac{\kappa}{8 |\lambda|}$, we get
\begin{align*}
\|\nabla u(t, \cdot) \|_{L^2}^2  + \| x  u(t, \cdot) \|_{L^2}^2\le & \ 2 E_\kappa (t)+\frac{2\lambda^2}{\kappa}\| u (t, \cdot)\|_{L^2}^2 \\
\le & \  2 E_\kappa(0)+\frac{2\lambda^2}{\kappa} \| u_0\|_{L^2}+C(\|u_0\|_{L^2}).
\end{align*}
Since the left hand side is the sum of two non-negative term, each of them is bounded individually by a uniform constant $C$, depending only on $ \| u_0 \|_\Sigma$.
\end{proof}

\subsection{Strichartz estimates} 
In order to prove global well-posedness of \eqref{eq:quintic_NLS} we shall heavily rely on the use of Strichartz estimates.
Let us briefly recall the definition and main properties of these estimates
for the following Schr\"odinger propagator
\begin{equation*}
U(t)=e^{-itH},\quad H:= -\frac{1}{2}\Delta +V(x),
\end{equation*}
where $V$ is given by \eqref{eq:pot}. The operator $U(t)$ consequently generates the linear time-evolution corresponding to \eqref{eq:diss_NLS}.
We first note that in view of Mehler's formula, cf. \cite{Carles, Car1}, the group $U(t)$ is not only bounded on
$L^2(\R^d)$, but also enjoys dispersive properties for small
time. More precisely it holds
\begin{equation}
\label{eq:dispest}
\lVert U(t)f \rVert_{L^\infty(\R^d)}\le \frac{C}{|t|^{d/2}}\lVert
f \rVert_{L^1(\R^d)},\quad \text{for }|t|\le \delta,
\end{equation}
for some $\delta=\delta(\omega_j)>0$, see \cite{Car1, Car2}. Note that for an harmonic potential as in
\eqref{eq:pot}, $\delta$ is
\emph{necessarily finite}, since $H$ has eigenvalues (see also below). 
From this dispersive estimates one is led to the following (local in-time) Strichartz estimates for $U(t)$ in terms of admissible index pairs.
\begin{definition}
A pair $(q,r)$ is admissible if $2\le r
\le\frac{2d}{d-2}$ (resp. $2\le r\le \infty$, if $d=1$ and $2\le r<
\infty$, if $d=2$)
and
$$\frac{2}{q}= d\left( \frac{1}{2}-\frac{1}{r}\right).$$ 
Then, for any space-time slab $I\times\R^d$, we can define the Strichartz norm
\begin{equation*}
\|f\|_{S^0(I\times\R^d)}:=\sup_{(q, r)}\|f\|_{L^q_tL^r_x(I\times\R^d)},
\end{equation*}
where the supremum is taken over all admissible pairs of exponents $(q, r)$.
\end{definition}
Following \cite{KT} (see also \cite{Car1, Car2}) one has the following estimates, where $(q',r')$ denotes the H\"older dual exponents of $(q,r)$:
\begin{lemma}\label{thm:strich}
Let $(q,r)$, $(q_1,r_1)$ and~$ (q_2,r_2)$ be admissible pairs. Let $I$
be some finite time interval. Then it holds
\begin{equation*}\label{eq:strich}
 \left \| U(\cdot)\varphi \right\|_{L^q(I;L^r)}\le C(r, d) |I|^{1/q} \| \varphi
 \|_{L^2},
\end{equation*}
and also
\begin{equation*}\label{eq:strichnl}
   \Big \| \int_{I\cap\{s\leq
   t\}} U(t-s)F(s) \, d s
   \Big \|_{L^{q_1}(I;L^{r_1})}\le C(r_1,r_2, d) |I|^{1/q_1}  \|
   F \|_{L^{q'_2}(I;L^{r'_2})} .
 \end{equation*}
\end{lemma}
\begin{proof} This result can be essentially be found in \cite[Propositon 3.3]{Car2}, but for the convenience of the reader we shall 
recall the basic idea in the case of an isotropic confinement, where the Hamiltonian is imply given by 
$$H= \frac{1}{2} \Delta + \frac{\omega^2}{2} {|x|^2},$$ 
with $\omega \in \R$. In this case, Mehler's formula yields local in-time dispersion on $|t| < \pi/(2\omega)$. 
Now, let $I$ be a given finite time-interval. Then we can split $I$ into (finitely many) sub-intervals $ I_j$, $j=1, \dots , N$ such that $|I_j |< \pi/(2\omega)$. 
Strichartz estimates (based on Mehler's formula) imply that on each $I_j$ it holds
$$ \| U(t) \varphi \|_{L^q(I_j; L^r)} \le C \| \varphi \|_{L^2},$$
for a universal $C=C(r,d)>0$.
From this we obtain
$$
\| U(t) \varphi \|^q_{L^q(I; L^r)} = \sum_{j = 1}^N\| U(t) \varphi \|^q_{L^q(I_j; L^r)}  \lesssim |I|\,  \| \varphi \|^q_{L^2}.
$$ 
This directly yields the first assertion of the lemma for any finite time-interval $I$ in the case of an isotropic potential. A similar argument can then be done in order to prove the second assertion stated above and a 
generalization to the case of a non-isotropic confinement is straightforward.
\end{proof} 

Taking $\varphi$ to be an eigenfunction of the (anisotropic) harmonic oscillator shows, that in general one can not expect the above given 
Strichartz-estimates to hold uniformly in-time, i.e. without dependence on the length of $I$, unless all the $\omega_j$ in \eqref{eq:pot} are in fact zero, 
in which case the operator $U(t)$ simplifies to the usual free Schr\"odinger group. In the upcoming section, this requires us to 
keep track of the dependence of all appearing constants on the length of $I$.

\subsection{Proof of Theorem \ref{th:critical}}

Having in mind the a-priori bounds 
obtained in Section \ref{sec:a-priori}, we can now state the proof of our main result.

\begin{proof}
We rewrite \eqref{eq:quintic_NLS} using Duhamel's formula
\begin{equation}\label{eq:duhamel}
u(t)=U(t) u_0 +i\l \int_0^tU(t-s)\(|u|^2 u\)(s) \, \D s
-\,  \sigma \int_0^tU(t-s)\(|u|^4 u\) (s) \, \D s .
\end{equation}
We aim to prove global well-posedness by a fixed point argument.
To this end, we first consider a space-time slab $I \times \R^3$, where $|I | < + \infty$, such that the $L^{10}_{t, x}$ norm of $u$ 
within this slab is small, say
\begin{equation}\label{eps}
\|u \|_{L^{10}_{t, x}(I\times\R^3)}\le \eps < 1.
\end{equation}
Then, using Strichartz and H\"older estimates, we can estimate \eqref{eq:duhamel} as follows:
\begin{align*}
\| u \|_{L^q_tL^r_x(I\times\R^3)}\lesssim & \, |I|^{1/q}\pare{\| u_0\|_{L^2}+\| |u |^2 u \|_{L^1_tL^2_x}
+\| | u |^4 u \|_{L^2_tL^{6/5}_x}}\\
\lesssim & \, |I|^{1/q}\pare{\| u_0\|_{L^2}+|I|^{1/2}\| u \|_{L^{10}_{t, x}}^2\| u \|_{L^{10/3}_{t, x}}
+\|u\|_{L^{10}_{t, x}}^4\|u\|_{L^{10}_tL^{30/13}_x}}
\end{align*}
and taking into account the smallness assumption stated above we get
\begin{align*}
\| u \|_{L^q_tL^r_x(I\times\R^3)}\lesssim \, |I|^{1/q}\pare{\| u _0\|_{L^2}+|I|^{1/2}\eps^2\| u \|_{S^0}+\eps^4\| u \|_{S^0}}.
\end{align*}
Next, in order to bound $\nabla u$ we first note that 
\begin{equation*}
[\partial_j,H]=  \partial_j V(x) ,\quad [x_j,H]
= \partial_j, \quad j=1, \dots, d.
\end{equation*}
where $[A,B]=AB-BA$ denotes the usual commutator. By assumption, $\partial_j V(x) = \omega_j^2 x_j$, i.e. linear in $x$. This shows that 
we can obtain a closed family of estimates for $\nabla u$ and $x u$. 
More precisely, we have 
\begin{align*}
\|\nabla u\|_{L^q_tL^r_x}+\|xu\|_{L^q_tL^r_x}
\lesssim &\, |I|^{1/q}\Big(\|\nabla u_0\|_{L^2}+\|xu_0\|_{L^2}+\||u|^2\nabla u\|_{L^1_tL^2_x}
+\|x|u|^2u\|_{L^1_tL^2_x}\\
& \, +\||u|^4\nabla u\|_{L^2_tL^{6/5}_x}+\|x|u|^4u\|_{L^2_tL^{6/5}_x}\Big)\\
\lesssim& \,  |I|^{1/q}\Big(\|\nabla u_0\|_{S^0}+\|xu_0\|_{S^0}+|I|^{1/2}\eps^2(\|\nabla u\|_{S^0}+\|xu\|_{S^0})\\
&+\eps^4(\|\nabla u\|_{S^0}+\|xu\|_{S^0})\Big).
\end{align*}
Thus, denoting the Strichartz norm in $\Sigma$ by
$$
\| u \|_{S_\Sigma}:=\| u \|_{S^0}+\|\nabla u \|_{S^0} + \| x u \|_{S^0},
$$
we infer 
\begin{equation*}\label{eq:103}
\| u \|_{S_\Sigma}\le \sup_{q} |I|^{1/q} \left( \|u_0\|_{\Sigma} + \eps^2 |I|^{1/2} \|u \|_{S_\Sigma} + \eps^4 \| u \|_{S_\Sigma} \right) .
\end{equation*}
Thus, if $\eps <1 $ defined in \eqref{eps}, is sufficiently small, a standard contraction argument yields
\begin{equation*}\label{eq:103}
\| u \|_{S_\Sigma}\le C(\|u_0\|_{\Sigma}, |I|).
\end{equation*}

Next, consider any finite time interval $I=[0,T]$, for some $T<+ \infty$. From Proposition \ref{prop:a-priori} 
we already know that the $L^{10}_{t, x}$ of $u$ is uniformly bounded but not necessarily small, say
$
\| u \|_{L^{10}_{t, x}(I\times\R^3)}\le M$, where $M$ is independent of the length of $I$.
Then we can divide $I$ into subintervals $I=I_1\cup\dotsc\cup I_N$, where $| I_\ell |  = C(\delta, \| u \| _{\Sigma})>0$, and such that the $L^{10}_{t,x}$ norm is sufficiently small in each $I_\ell =[t_{\ell-1}, t_\ell]$, i.e.
\begin{equation*}
\|u\|_{L^{10}_{t, x}(I_\ell \times\R^3)}\le \eps, \quad \textrm{for all $\ell =1, \dotsc, N$.}
\end{equation*}
Note that for any $I=[0,T]$ this $\eps <1$ only depends on $\|u_0 \|_{\Sigma}$ and not on the length of the interval $I$, in view of Proposition \ref{prop:a-priori}. 

By the same fixed point argument as before, we consequently obtain that in each $I_\ell \times \R^3$ it holds
\begin{equation*}
\|u\|_{S_\Sigma(I_\ell \times \R^3)}\le C(\|u(t_{\ell -1},\cdot) \|_{\Sigma}, |I_\ell|),\quad \ell =1, \dotsc, N.
\end{equation*}
By Corollary \ref{cor:a-priori} we also have that $\| u(t, \cdot)\|_{\Sigma}$ is uniformly bounded for all $t \ge 0$ and thus
\begin{equation*}
\|u\|_{S_\Sigma(I_\ell \times\R^3)}\le C(\|u_0\|_{\Sigma}, |I_\ell |).
\end{equation*}
Summing up all the subintervals $I_\ell$ we consequently infer
\begin{equation*}
\|u\|_{S_\Sigma(I\times\R^3)}\le C(\|u_0\|_{\Sigma}, M).
\end{equation*}
By continuity, we consequently obtain a unique solution $u$ in $[0,T]\times \R^3$, for any $T\in (0,\infty)$ and thus 
we conclude that the Cauchy problem \eqref{eq:quintic_NLS} is globally well-posed in $\Sigma$ (depending continuously on the initial data). 
\end{proof}

As a by-product of our analysis we infer that all Strichartz norms are uniformly bounded during the time evolution, at least in situations without harmonic confinement. 
\begin{proposition}
Let $d=3$, $p=5$ and assume $V(x)\equiv 0$. Then for any admissible pair of exponents $(q,r)$ it holds:
\begin{equation*}
\| u \|_{L^q([0, \infty);L^r(\R^3))}+\|\nabla u \|_{L^q([0, \infty);L^r(\R^3))}
\le C(\|u_0\|_{L^2}, E_\kappa(0)).
\end{equation*}
\end{proposition}
\begin{proof}
Let $(q, r)$ be an arbitrary admissible pair of exponents, $I$ an arbitrary time interval (which could also be infinite), and let $t^\ast\in I$. 
Then, by Strichartz estimates we have
\begin{multline*}
\|\nabla u \|_{L^q_tL^r_x(I\times\R^3)}\lesssim\|\nabla u (t^\ast)\|_{L^2(\R^3)}
+\||u |^2\nabla u \|_{L^{10/7}_{t, x}(I\times\R^3)}
+\||u |^4\nabla u \|_{L^{10/7}_{t, x}(I\times\R^3)},
\end{multline*}
where we recall that $(10/3, 10/3)$ is an admissible pair of exponents. By Hölder's inequality we obtain
\begin{align*}\label{eq:strich1}
\|\nabla u \|_{L^q_tL^r_x(I\times\R^3)}\lesssim \|\nabla u (t^\ast)\|_{L^2(\R^3)}
+\|u\|_{L^{10/3}_{t, x}}\|u \|_{L^{10}_{t, x}}\|\nabla u \|_{L^{10/3}_{t, x}} +\|u \|_{L^{10}_{t, x}}^4\|\nabla u \|_{L^{10/3}_{t, x}}.
\end{align*}
Analogously, we obtain
\begin{equation*}\label{eq:strich2}
\|u \|_{L^q_tL^r_x(I\times\R^3)}\lesssim\|u(t^\ast)\|_{L^2(\R^3)}
+\|u\|_{L^{10}_{t, x}}\|u\|_{L^{10/3}_{t, x}}^2
+\|u\|_{L^{10}_{t, x}}^4\|u\|_{L^{10/3}_{t, x}}.
\end{equation*}
Adding the last two inequalities and taking the supremum over all admissible pairs $(q, r)$ we infer 
\begin{equation*}\label{eq:strich_tot}
\|u\|_{S^1(I\times\R^3)}\lesssim\|\nabla u(t^\ast)\|_{L^2(\R^3)}
+\|u\|_{L^{10}_{t, x}}\|\psi\|_{S^1}^2\\
+\|u\|_{L^{10}_{t, x}}^4\|u\|_{S^1}.
\end{equation*}
Since we already know that $\|u\|_{L^{10}_{t, x}} \le C(\|u_0\|_{L^2}, E_\kappa(0))$, we can divide the time-interval $[0, T]$ for any $T>0$ into subintervals $I_j$, $j=1, \dotsc, N$, such that
\begin{equation*}
\|u\|_{L^{10}_{t, x}(I_j\times\R^3)}\le\eps,\qquad\forall\;j=1,\dotsc, N.
\end{equation*}
Note that $N$ only depends on $\eps>0$ and on the constant $C(\|u_0\|_{L^2}), E_\kappa(0))$. Hence, for $\eps$ sufficiently small a standard bootstrap argument yields 
\begin{equation*}
\|u\|_{S^1(I_j\times\R^3)}\le C(\|u_0\|_{L^2}, E_\kappa(0))\|\nabla u(t_j)\|_{L^2(\R^3)},
\end{equation*}
where $t_j\in I_j$. Since we also know that $\|\nabla u(t)\|_{L^2}\le C(\|u_0\|_{L^2}, E_\kappa(0))$ for each $t>0$, we conclude
\begin{equation*}
\|u\|_{S^1(I_j\times\R^3)}\le C(\|u_0\|_{L^2}, E_\kappa(0)),\qquad\forall\;j=1,\dotsc,N.
\end{equation*}
Thus, by summing over the $N$ subintervals we obtain the desired result.
\end{proof}
\section{Proofs for subcritical damping and the time-decay of solutions}\label{sec:sub}

This Section is devoted to the remaining proofs for Corollaries \ref{cor:sub} - \ref{cor:decay}.
To this end, we first note that in the case of an energy-subcritical damping term, one easily concludes 
local in-time well-posedness in $\Sigma$ by
classical arguments, see \cite{Caz, T, Car1, Car2, Zh}. Indeed we have the following blow-up alternative:

\begin{lemma} \label{lem:blow-up} Let $V$ be a quadratic confinement of the form \eqref{eq:pot} and $u_0 \in \Sigma$.
Moreover, let $p\ge 3$, if $d=1,2$ and $3\le p< 5$, if $d=3$. Then there exists a unique local in-time solution
$u \in C([0, T), \Sigma)$. Moreover, if $T< + \infty$, then 
$$\lim_{t \to T} \| u (t, \cdot) \|_{\Sigma} = \infty,$$
\end{lemma}
In order to continue this local-in-time solution for all times we again need to derive suitable a-priori estimates. This will be done in the following lemma.

\begin{lemma}\label{lem:sub}
Let $V$ be a quadratic confinement of the form \eqref{eq:pot} and $u(t) \in \Sigma$ be a solution to \eqref{eq:diss_NLS}. 
Then, if either:
\begin{itemize} 
\item 
$p> 3$, for $d=1,2$, respectively $3<p\le 5$, for $d=3$,
\item or $p=3$ and $\sigma \ge \text{\rm max}\, \{0, -\lambda\}$,
\end{itemize}it holds
\begin{equation*}
\| u(t, \cdot) \|_{\Sigma} \le C(\| u_0 \|_{\Sigma}),\quad \forall \, t \ge 0.
\end{equation*}
\end{lemma}
Combining this uniform bound on the energy norm of $u$ with the assertion of Lemma \ref{lem:blow-up}, consequently proves Corollary \ref{cor:sub} and \ref{cor:cubic}.
\begin{proof} 
We first consider the case where $p>3$ and consider the following energy-type functional
\begin{equation*}\label{eq:diss_NLS_enP}
E_{\kappa, p}(t):=\int_{\R^d}\frac{1}{2}\, |\nabla u(t,x)|^2+ V(x) |u(t,x)|^2+ \frac{\lambda}{2}\, |u(t,x)|^4+\kappa|\psi(t,x)|^{p+1} d x,
\end{equation*}
with $\kappa >0$ to be chosen later on. The time-derivative of  $E_{\kappa, p}(t)$ is then found to be
\begin{align*}
\frac{d}{d t} \, E_{\kappa, p}(t)=
&-\sigma(p-1)\int_{\R^d}|u|^{p-1}|\nabla\sqrt{\rho}|^2\, d x-\sigma\int_{\R^d}|u|^{p-1}|\nabla u|^2\, d x\\
&+\frac{\kappa}{4}(p+1)(p-1)\int_{\R^d}|u|^{p-3}\nabla\rho\cdot J\, d x
-2\sigma\int_{\R^d} V|u|^{p+1}\,dx\\
& -\sigma \kappa(p+1)\int_{\R^d}|u|^{2p}\, d x -2\lambda\sigma\int_{\R^d}|u|^{p+3}\, d x.
\end{align*}
We rewrite, similarly as before,
\begin{equation*}
\int_{\R^d}|u|^{p-3}\nabla\rho\cdot J\, d x
=-\int_{\R^d}|u|^{p-3}\modu{\mez\nabla\rho-J}^2\, d x
+\int_{\R^d}|u|^{p-1}|\nabla u|^2\, d x,
\end{equation*}
and also use the interpolation estimate
\begin{equation*}
\|u \|_{L^{p+3}}^{p+3}\le\|u\|_{L^{p+1}}^{\frac{(p-3)(p+1)}{p-1}} \,
\|u \|_{L^{2p}}^{\frac{4p}{p-1}}.
\end{equation*}
Using this, and following the arguments given in the proof of Proposition \ref{prop:a-priori}, we can 
always find a $\kappa = \kappa(d, p)>0$ such that 
\begin{equation*}
E_{\kappa, p}(t)\le E_{\kappa, p}(0)+C_p (\|u_0\|_{L^2}), \quad \forall \, t\ge 0.
\end{equation*}
This consequently implies a uniform bound on $\| u(t, \cdot) \|_{\Sigma}$ by a interpolation arguments similar to those given in the proof of Corollary \ref{cor:a-priori}.

In a second step, we turn to the threshold situation $p=3$: In the defocusing case $\lambda >0$ one can 
use the energy-type functional specified above with $\kappa = 0$ and $p=3$. One analogously proves that $E_{0,3}(t)$  
is decreasing along solution $u(t)$ and consequently concludes that $\| u(t, \cdot) \|_{\Sigma}$ is uniformly bounded. 

In the focusing case $\lambda <0 $ we have to argue slightly differently. 
Recall that in this case, the NLS type equation can be written as
\begin{equation}\label{eq:diss_cubic1}
i\d_t u =- \mez\Delta u +V(x) u -(|\lambda|+i\sigma)|u |^2 u .
\end{equation}
We consequently consider the corresponding linear energy functional 
\begin{equation*}\label{eq:diss_NLS_lin}
E_{\rm lin}(t):=\int_{\R^d}\frac{1}{2}\, |\nabla u(t,x)|^2+ V(x) |u(t,x)|^2d x.
\end{equation*}
Differentiating $E_{\rm lin}(t)$ w.r.t time yields
\begin{align*}
\frac{d}{d t} \, E_{\rm lin}(t)=-\int_{\R^d}\re(\Delta\conj{u } (t,x) \d_t u (t,x))\, d x - 2 \sigma \int_{\R^d} V(x) |u|^4 \, dx
\end{align*}
and using equation \eqref{eq:diss_cubic1} we obtain
\begin{align*}
\frac{d}{d t} \,  E_{\rm lin}(t)
= & \ \sigma\int_{\R^d}| u |^2\re(\conj{ u }\Delta u )\, d x  +|\lambda| \int_{\R^d}| u |^2\im( u \Delta\conj{ u })\, d x - 2 \sigma \int_{\R^d} V(x) |u|^4 \, dx\\
\le & -\frac{\sigma}{2}\int_{\R^d}|\nabla\rho|^2\, d x - \sigma\int_{\R^d}|\conj{u}\nabla u  |^2\, d x
+ 2|\lambda| \int_{\R^d}\re(\conj{u }\nabla u)\cdot\im(\conj{ u }\nabla u )\, d x.
\end{align*}
This can be re-written as
\begin{align*}
\frac{d}{d t} \,  E_{\rm lin}(t)\le & \, -\frac{\sigma}{2}\int_{\R^d}|\nabla\rho|^2\, d x -(\sigma-|\lambda| )\int_{\R^d}\rho|\nabla u |^2\, d x \\
& \, -|\lambda| \int_{\R^d}\modu{\frac{1}{2}\re(\conj{u }\nabla u )- \im(\conj{u }\nabla  u )}^2\, d x,
\end{align*}
and thus, if $\sigma\ge |\lambda|$, we consequently obtain $E_{\rm lin}(t) \le E_{\rm lin}(0) < \infty$. This 
yields global well-posedness of the considered NLS \eqref{eq:diss_cubic1} and we are done.
\end{proof}

\begin{remark} \label{rem: blow-up} Note that the blow-up alternative given in Lemma \ref{lem:blow-up} is not exactly the same as in the case of the usual NLS (without damping), 
for which it is enough to control $ \|\nabla u(t, \cdot)\|_{L^2}$. 
To see this, consider the following focussing NLS (cubic, for simplicity, but a generalization to other power-type nonlinearities is straightforward):
\begin{equation*}
i\d_t u =- \mez\Delta u +V(x) u -  |u |^2 u .
\end{equation*}
It is well known that this equation admits local in-time solutions which in addition preserve the mass $M(t)= \|u(t)\|_{L^2}$ and the energy functional 
$$
E (t)=\int_{\R^d}\frac{1}{2}\, |\nabla u(t,x)|^2+ V(x) |u(t,x)|^2 - \frac{ 1}{2} |u |^4 d x.
$$
Using the Gagliardo-Nirenberg inequality
\begin{equation*}
 \lVert u \rVert_{L^4}^4 \le C  \lVert u \rVert_{L^2}^{4-d} \, \lVert
 \nabla u
 \rVert_{L^2}^d,
\end{equation*}
one observes that, in view of mass conservation, the (conserved) energy is a sum of three terms, two of which are bounded, provided that $ \|\nabla u(t, \cdot)\|_{L^2}< \infty$. 
Thus, also the third term, i.e. the linear potential energy $\propto  \|x u(t, \cdot)\|_{L^2}$, has to be bounded. This shows that unless $\lVert \nabla
u (t, \cdot) \rVert_{L^2}$ becomes unbounded, $\| u (t, \cdot)\|_{\Sigma}$ is a
continuous function in time. In other words, even in the case of a given quadratic confinement, the possible blow-up of solutions can essentially be regarded as a local phenomena in $x$ 
(analogously to the case without potential).

Our situation is a bit more involved, though: As we have seen, finite-time blow-up is only possible for $\lambda <0$ and $p= 3$. 
In this case, however, we are only able to derive boundedness of $E_{\rm lin}(t)$, provided that $\sigma < |\lambda|$. We therefore do not get any further insight on the nature of the blow-up and 
can not rule out the possibility that the $H^1(\R^3)$ norm of the solution stays bounded but $\| x u (t, \cdot)\|_{L^2} \to  \infty$ (even though such a situation seems to be unlikely).
\end{remark}

\begin{remark} A closely related observation concerns the following: If instead of \eqref{eq:pot}, we would consider a \emph{repulsive} potential of the form $V(x) = - \frac{1}{2} |x|^2$,
we would not succeed with our approach. More precisely, in the case of a repulsive potential,
the corresponding term in the time-derivative of $E_{\kappa, p}(t)$ comes with the wrong sign and thus we can not conclude that $E_{\kappa,p}(t)$ is non-increasing.
This is remarkable insofar as it is known that (sufficiently strong) repulsive quadratic potentials are an obstruction for the possible blow-up of solutions to NLS, see \cite{Carles}. 
In our case, the situation is not so clear and global well-posedness for repulsive potentials remains an interesting open problem.
\end{remark}

Having set up global well-posedness of the equation \ref{eq:diss_NLS}, we finally
turn to the proof of the decay of its solutions $u(t,x)$ as $t\to + \infty$.

\begin{proof}[Proof of Corollary \ref{cor:decay}]
Let $u\in C([0, \infty); \Sigma)$ be the global solution to \eqref{eq:diss_NLS}. From the
dissipation equation \eqref{eq:diss_rho} we obtain that
\begin{equation*}
\int_0^\infty\int_{\R^n}|u|^{p+1}\,dx\,dt\le C(\|u_0\|_{L^2}).
\end{equation*}
Consider a sequence of time-steps $(t_n)_{n\in \N}$, tending to infinity and define
\begin{equation*}
u_n(t, x):=u(t+t_n, x).
\end{equation*}
From the space-time bound given above, we consequently infer that, as $n \to \infty$:
$$u_n\to0, \quad  \text{in $L^{p+1}([0, \infty)\times\R^d)$}.
$$
On the other hand, from our global in-time existence theory, we know that $(u_n)_{n \in \N}$ is uniformly bounded in
$ C([0, \infty); H^1(\R^d))$ and that $(\d_tu_n)_{n \in \N}$ is uniformly bounded in $C([0, \infty); H^{-1}(\R^d))$.
Thus by the Aubin-Lions Lemma we conclude that $(u_n)_{n \in \N}$ is relatively compact in $C([0, T]; L^2(\R^d))$, for each $0<T<\infty$.
We consequently have that there exists a subsequence $(u_{n_\ell})$, such that
\begin{equation*}
u_{n_\ell}\to0\qquad\textrm{in}\; C([0, T]; L^2(\R^d)).
\end{equation*}
From this, we consequently obtain \eqref{decay}, since
\begin{equation*}
\|u(t_{n_\ell})\|_{L^2}\le\sup_{0\le t\le T}\|u_{n_\ell}(t)\|_{L^2}\to0.
\end{equation*}
In addition, by using the smoothing properties of the Schr\"odinger group $U(t)$, see \cite{CoSa}, we 
get that $(u_n)_{n \in \N}$ is uniformly bounded in $L^2((0,T); H^{3/2}_{\rm loc}(\R^d))$.
Invoking a compactness argument given in \cite{RaTe} we infer that $(u_n)_{n \in \N}$ is pre-compact in $L^2((0,T); H^{1}_{\rm loc}(\R^d))$, 
which concludes the proof. 
\end{proof}

{\bf Acknowledgement:} The authors want to thank R. Carles for helpful discussions.


\begin{thebibliography}{amsplain}

\bibitem{A}  S. K. Adhikari, \emph{Mean-field description of collapsing and exploding Bose-Einstein condensates}, Phys. Rev. A, {\bf 66} (2002), issue 1, 13611--13619.
\bibitem{AK} I. S. Aranson and L. Kramer, \emph{The world of the complex Ginzburg-Landau equation}, Rev. Mod. Phys. {\bf 74} (2002), 99--143.
\bibitem{AnMa} P. Antonelli and P. Marcati, \emph{On the finite energy weak solutions to a system in quantum fluid dynamics}, Comm. Math. Phys. {\bf 287} (2009), no 2, 657--686.
\bibitem{BJ} W. Bao and D. Jaksch, \emph{An explicit unconditionally stable numerical method for solving damped nonlinear Schr\"odinger equations with a focusing nonlinearity}.
SIAM J. Numer. Anal. {\bf 41} (2003), no. 4, 1406--1426.
\bibitem{BJM} W. Bao, D. Jaksch, and P. Markowich, \emph{Three dimensional simulation of jet formation in collapsing condensates}.
J. Phys. B: At. Mol. Opt. Phys. {\bf 37} (2004), no. 2, 329--343.
\bibitem{Bi} E. A. Biswas, \emph{Optical soliton perturbation with nonlinear damping and saturable amplifiers}.
Math. Comput. Simul. {\bf 56} (2001), issue 6, 521--537.
\bibitem{Car1}  R. Carles, \emph{Remarks on nonlinear Schr\"odinger equations with harmonic potential}. Ann. Henri Poincare {\bf 3} (2002), 757--772.
\bibitem{Car2} R. Carles, \emph{Semi-classical Schrödinger equations with harmonic potential and nonlinear perturbation}. Annales I.H.P., Analyse non lin\'eaire {\bf 20} (2003), no. 3, 501--542.
\bibitem{Car3}  R. Carles, \emph{Nonlinear Schrödinger equations with repulsive harmonic potential and applications}. SIAM J. Math. Anal. {\bf 35} (2003), no. 4, 823--843. 
\bibitem{Car}  R. Carles, \emph{Linear vs. nonlinear effects for nonlinear Schr\"odinger equations with potential}.
Commun. Contemp. Math. {\bf 7} (2005), no. 4, 483--508.
\bibitem{Carles}  R. Carles, \emph{Global existence results for nonlinear Schr\"odinger equations with quadratic potentials}, Discrete Contin. Dyn. Syst. {\bf 13} (2005), no. 2, 385--398.
\bibitem{Caz} T. Cazenave, \emph{Semilinear Schr\"odinger Equations}.
Courant Lecture Notes in Mathematics vol. 10, New York University, Courant Institute of Mathematical Sciences, AMS, 2003.
\bibitem{CP} T. Chen and  N. Pavlovic, \emph{The quintic NLS as the mean field limit of a Boson gas with three-body interactions.}
Preprint {\tt arXiv:0812.2740v1}.
\bibitem{CKSTT} J. Colliander, M. Keel, G. Staffilani, H. Takaoka, and T. Tao, \emph{Global well-posedness and scattering for the energy-critical nonlinear Schr\"odinger equation in $\Bbb R\sp 3$}.
Ann. of Math. (2) {\bf 167} (2008), no. 3, 767--865.
\bibitem{CoSa} P. Constantin and J. C. Saut, \emph{Local Smoothing Properties of Dispersive Equations}. 
J. Amer. Math. Soc. {\bf 1} (1988), 413-?439. 
\bibitem{Fi} G. Fibich, \emph{Self focusing in the damped nonlinear Schr\"odinger equation}.
SIAM J. Appl. Math. {\bf 61} (2001), 1680--1705.
\bibitem{Fu} D. Fujiwara, \emph{A construction of the fundamental solution for the Schr\"odinger equation}. J. Analyse Math {\bf 35} (1979), 4--96
\bibitem{GaMa} I. Gasser, P. A. Markowich, \emph{Quantum hydrodynamics, Wigner transforms and the classical limit}. Asympt. Anal. {\bf 14} (1997), 97--116.
\bibitem{Go} O. Goubet, \emph{Asymptotic smoothing effect for a weakly damped nonlinear Schr\"odinger equation in $T\sp 2$}.
J. Diff. Equ. {\bf 165} (2000), no. 1, 96--122.
\bibitem{JuPi} A. J\"ungel and R. Pinnau, \emph{Inviscid limits of the complex Ginzburg-Landau equation}. Comm. Math. Phys. {\bf 214} (2000) 201--226. 
\bibitem{KMS} Y. Kagan, A. E. Muryshev, and G. V. Shlyapnikov, \emph{Collapse and Bose-Einstein Condensation in a Trapped Bose Gas with Negative Scattering Length}.
Phys. Rev. Lett. {\bf 81} (1998), issue 5, 933--937.
\bibitem{KT} M. Keel and T. Tao, \emph{Endpoint Strichartz Estimates}.
Amer. J. Math. \textbf{120} (1998), 955--980.
\bibitem{KiVi} R. Killip and M. Visan, \emph{Energy-critical NLS with quadratic potentials}. Comm. Partial Diff. Equ. to appear.
\bibitem{KiShi} N. Kitaa and A. Shimomura, \emph{Asymptotic behavior of solutions to Schr\"odinger equations with a subcritical dissipative nonlinearity}. J. Diff. Equ. {\bf 242} (2007), issue 1, 192--210.
\bibitem{La} P. Laurencot, \emph{Long-time behaviour for weakly damped driven nonlinear Schr\"odinger equations in $R\sp N$, $N\leq 3$}.
NoDEA Nonlin. Diff. Equ. Appl. {\bf 2} (1995), no. 3, 357--369.
\bibitem{Oh}  Y-G. Oh, \emph{Cauchy problem and Ehrenfest's law of nonlinear Schr\"odinger equations with potentials}. J. Diff. Equ. {\bf 81} (1989), 255--274.
\bibitem{OT} M. Ohta and G. Todorova, \emph{Remarks on global existence and blow-up for damped nonlinear Schr\"odinger equations}.
Discrete Contin. Dyn. Syst. {\bf 23} (2009), no. 4, 1313--1325.
\bibitem{PSS} T. Passot, C. Sulem, and P. L. Sulem, \emph{Linear versus nonlinear dissipation for critical NLS equation}.
Physica D {\bf 203} (2005), issue 3-4, 167--184.
\bibitem{PeSt} N. R. Pereira and L. Stenflo, \emph{Nonlinear Schr\"odinger equation including growth and damping}.
Phys. Fluids {\bf 20} (1977), 1733--1734.
\bibitem{RaTe} J.-M. Rakotoson and R. Temam, \emph{An optimal compactness theorem and application to elliptic-parabolic systems}. 
Appl. Math. Lett. {\bf 14} (2001), no. 3, 303--306.
\bibitem{SU} H. Saito, M. Ueda, \emph{Intermittent implosion and pattern formation of trapped Bose-Einstein condensates with attractive interaction.} Phys. Rev. Lett. {\bf 86} (2001), 1406--14011.
\bibitem{SaMa} H. Sakaguchi and B. A. Malomed, \emph{Two-dimensional dissipative gap solitons}. Preprint {\tt ariXiv:0908.0973v1}.
\bibitem{Shi} A. Shimomura, \emph{Asymptotic Behavior of Solutions for Schr\"odinger Equations with Dissipative Nonlinearities}, Comm. Part. Diff. Equ.  {\bf 31} (2006), issue 9, 1407--1423.
\bibitem {SuSu} C. Sulem and P.~L. Sulem, \emph{The nonlinear Schr\"odinger equation},
Applied Math. Sciences 139, Springer 1999.
\bibitem{T} T. Tao, \emph{Nonlinear Dispersive Equations: Local and Global Analysis}.
CBMS Regional Conference Series in Mathematics, AMS 2006.
\bibitem{TVZ} T. Tao, M. Visan, and X. Zhang, \emph{The Nonlinear Schrödinger Equation with Combined Power-Type Nonlinearities}. Comm. Part. Diff. Equ. {\bf 32}, issue 8, 2007, 1281--1343.
\bibitem{Ts1} M. Tsutsumi, \emph{Nonexistence of global solutions to the Cauchy problem for the damped nonlinear Schr\"odinger equations}.
SIAM J. Math. Anal. {\bf 15} (1984), no. 2, 357--366.
\bibitem{Zh} J. Zhang, \emph{Stability of attractive Bose-Einstein condensates}, J. Statist. Phys. {\bf 101} (2000), no. 3/4, 731--746.
\end{thebibliography}
\end{document}